\documentclass[reqno,12pt,letterpaper]{amsart}

\RequirePackage{amsmath,amssymb,amsthm,graphicx,mathrsfs,url}
\RequirePackage[usenames,dvipsnames]{color}
\RequirePackage[colorlinks=true,linkcolor=Red,citecolor=Green]{hyperref}
\RequirePackage{amsxtra}
\usepackage{verbatim}
\usepackage{subfigure}

\relax

\numberwithin{equation}{section}

\def\arXiv#1{\href{http://arxiv.org/abs/#1}{arXiv:#1}}

\newtheorem{theo}{Theorem}
\newtheorem{prop}{Proposition}[section]

\newtheorem{lemm}[prop]{Lemma}

\def\Remark{\noindent\textbf{Remark.}\ }
\def\Remarks{\noindent\textbf{Remarks.}\ }

\let\Im=\Imag 

\DeclareMathOperator{\Op}{Op}

\let\Re=\Real 

\DeclareMathOperator{\supp}{supp}

\DeclareMathOperator{\WF}{WF}

\newcommand{\RR}{{\mathbb R}}
\newcommand{\NN}{{\mathbb N}}
\newcommand{\CC}{{\mathbb C}}

\title[Fermi Golden Rule]%
{Dynamics of resonances for 0th order pseudodifferential operators}
\author{Jian Wang }
\email{wangjian@berkeley.edu}
\address{Department of Mathematics, University of California, Berkeley,
CA 94720}

\begin{document}

\begin{abstract}
We study the dynamics of resonances of analytic perturbations of 0th order pseudodifferential operators $P(s)$. In particular, we prove a Fermi golden rule for resonances of $P(s)$ at embedded eigenvalues of $P=P(0)$. We also study the dynamics of eigenvalues of $P(t)=P+it\Delta$ as the eigenvalues converge to simple eigenvalues of $P$. The 0th order pseudodifferential operators we consider satisfy natural dynamical assumptions and are used as microlocal models of internal waves.
\end{abstract}

\maketitle

\section{Introduction}

In this note, we are interested in the dynamics of the resonances of 0th order pseudodifferential operators. 0th order pseudodifferential operators arise naturally in the study of fluid, in particular, the study of internal waves. 
\begin{figure}[hb]
    \centering
    \subfigure[Resonances $\lambda(s)$ of $P(s)$ near $0$]{
    \includegraphics[scale=0.3]{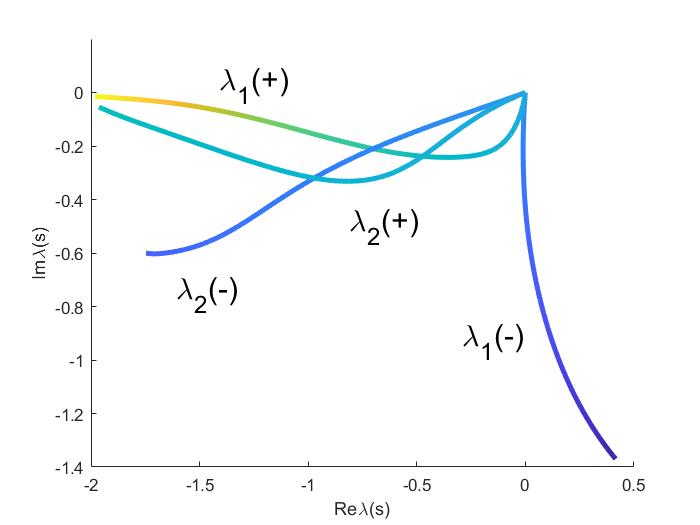}}
    \subfigure[$\Im\lambda(s)$]{
    \includegraphics[scale=0.3]{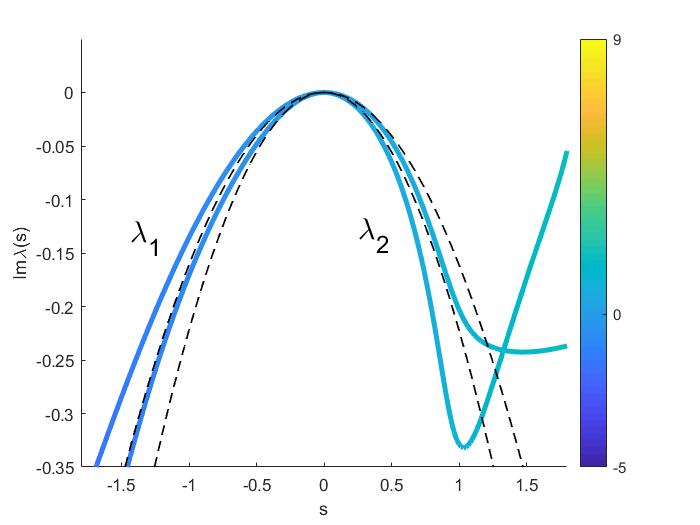}}
    \caption{Resonances of $P(s)$ near $\lambda_0=0$, where $P(s)$ is of the form \eqref{pex} with $V_m$, $V_a$ given by \eqref{multiv}, \eqref{multivv}. In this example, $0$ is an eigenvalue of $P$ with multiplicity $2$ and eigenfunctions $e^{\pm ix_1}/2\pi$. \textbf{(a)}: $\lambda_{1,2}(\pm)$ denotes resonances of $P(s)$ with $\pm s>0$, $-6.2<s<9.2$. \textbf{(b)}: $\Im\lambda_{1,2}(s)$ with small $s$. Dashed lines are the theoretic approximation of $\Im\lambda_{1,2}(s)$, see \eqref{multiimap}.}
    \label{multi}
\end{figure}
We refer to \cite{ralston} for the early work. Colin de Verdi\`ere and Saint-Raymond \cite{attractor} used 0th order pseudodifferential operators with dynamical assumptions on 2 dimensional torus as microlocal model for internal waves. Colin de Verdi\`ere \cite{co2} generalized the results to manifolds of higher dimensions. Dyatlov and Zworski \cite{force} provided alternative proofs of the results in \cite{attractor} using tools from scattering theory. Wang \cite{scattering} studied the scattering matrix for these operators.

In the first part of this paper, we consider the perturbation theory for a 0th order pseudodifferential operator $P$. We consider the case when $P$ has embedded eigenvalues $\lambda$. If $P(s)$ is a family of 0th order operators with $P=P(0)$, under certain conditions, the resonances of $P(s)$ near $\lambda$ converge to $\lambda$. If $\lambda$ has multiplicity $m>1$, we show the resonances of $P(s)$ allow expansions as Puiseux series. In the case when $\lambda$ is a simple eigenvalue of $P$, we propose and prove a Fermi golden rule -- for references of Fermi golden rules, we refer to Simon \cite{bs} for $n$-body quantum systems; Colin de Verdi\'ere \cite{co1} for the generic absence of embedded eigenvalues for surfaces with variable curvature and cusps; Phillips and Sarnak \cite{cusp} for the Laplacian operator on automorphic functions; Lee and Zworski \cite{quangraph} for quantum graphs; \cite[Theorem 4.22]{res} for a textbook style presentation of Fermi golden rule for black box scattering.

We are also interested in embedded eigenvalues $\lambda$ as limits of eigenvalues $\lambda(t)$ of $P(t)=P+it\Delta$. Galkowski and Zworski \cite{vis} defined the set of resonances of $P$ and showed the resonances can be approximated by the eigenvalues of $P(t)$. In the case when $\lambda$ is simple, we compute the first derivative of $\lambda(t)$ at $0$.

\subsection{Main results.} Let $\mathbb{T}^n=\RR^n/(2\pi\mathbb{Z})^n$ be the torus and $P(s)\in \Psi^0(\mathbb{T}^n)$ be a family of 0th order self-adjoint pseudodifferential operator on $\mathbb{T}^n$ defined by
\begin{equation}
\label{assumption1}
    P(s)u(z):=\frac{1}{(2\pi )^n}\int_{\RR^n}\int_{\RR^n}e^{i\langle z-z^{\prime},\zeta \rangle}p(s,z,\zeta)u(z^{\prime})dz^{\prime}d\zeta.
\end{equation}
Here $(s,z,\zeta)\in (-s_0,s_0)\times T^*\mathbb{T}^n$ for some $s_0>0$ and $p(s)\in S^0(T^*\mathbb{T}^n)$ is called the full symbol of $P(s)$. In the integral \eqref{assumption1} we view $p$ as a $2\pi$ periodic function in $z$ over $\RR^n$ and the integral is considered in the sense of oscillatory integrals (see \cite[\S 5.3]{semi}).
We assume that 
\begin{equation}
\label{assumption2}
    \text{$p(s)$ is analytic in $s$,}
\end{equation}
and for each $s\in(-s_0,s_0)$,
\begin{equation}\begin{split}
\label{assumption3}
    \text{$p(s)$ has an analytic continuation from $T^*\mathbb{T}^n$ to $\CC^{2n}$ such that}\\  |p(s,z,\zeta)|\leq M < \infty, \quad \text{ for } |\Im z|\leq a_1, \quad |\Im\zeta|\leq a_2\langle \Re \zeta \rangle
\end{split}\end{equation}
with $a_1, a_2>0$.
We also assume that  
\begin{equation}\begin{split}
\label{assumption4}
    & \text{there exists $F\in S^1(T^*\mathbb{T}^n)$, $C>0$ such that}\\
    & H_{p(s)}F(z,\zeta)>0, \quad \text{for } (z,\zeta)\in \{p(s,z,\zeta)=0\}\cap\{ |\zeta|>C \}.
\end{split}\end{equation}

\begin{figure}[ht]
    \centering
    \subfigure[Resonances $\lambda(s)$ of $P(s)$ near $0$]{
    \includegraphics[scale=0.3]{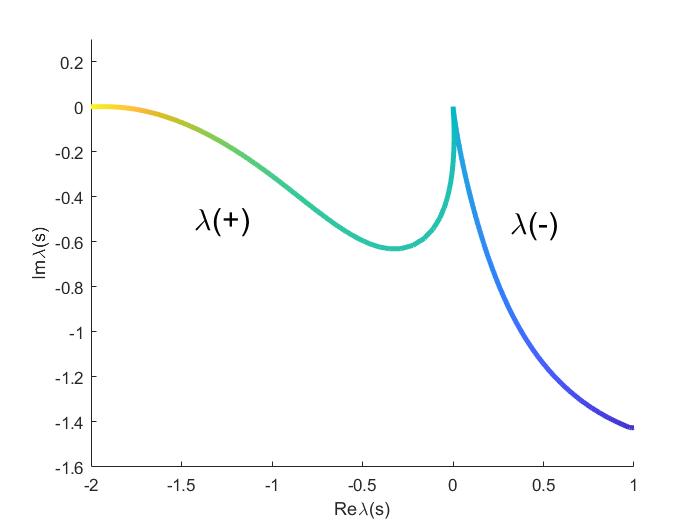}}
    \subfigure[$\Im\lambda(s)$]{
    \includegraphics[scale=0.3]{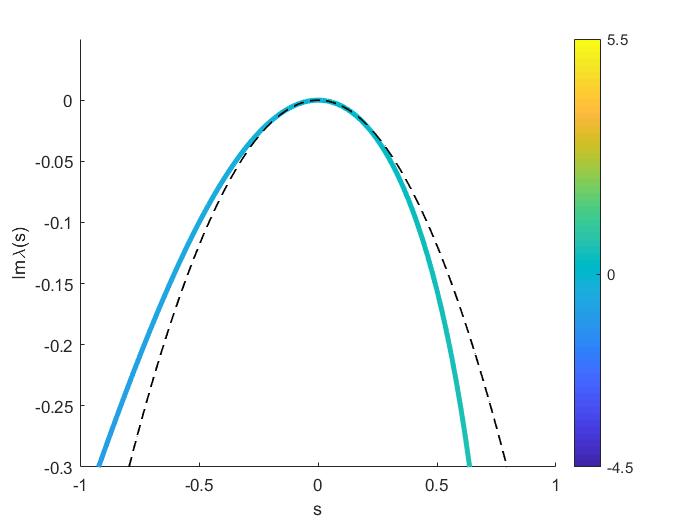}}
    \caption{Resonances of $P(s)$ where $P(s)$ is of the form \cite[(B.4)]{vis} with $V_m$, $V_a$ given by \eqref{simplev}, \eqref{simplevv}. In this case $0$ is a simple eigenvalue of $P$ with eigenfunction $e^{ix_1}/2\pi$. \textbf{(a)}: $\lambda(\pm)$ denotes the resonances of $P(s)$ with $\pm s>0$, $-4.5<s<5.5$. \textbf{(b)}: $\Im\lambda(s)$ with $s$ small. The dashed line is the theoretic approximation of $\Im\lambda(s)$, see \eqref{simpleimap}.}
    \label{simple}
\end{figure}

Motivated by definitions in \cite{hs} and \cite{den}, let $H_{\Lambda}$ be spaces come from \cite[(4.7)]{vis}, see also \S \ref{hspace}. Then by \cite[Lemma 7.4]{vis}, there exists $\zeta_0>0$ such that if $\Im \zeta>-\zeta_0 \theta$ and $|\Re \zeta|<\zeta_0$
\begin{equation}
    P(s)-\zeta: H_{\Lambda} \to H_{\Lambda}
\end{equation}
is a Fredholm operator, where $H_{\Lambda}$ is the space of hyperfunctions, see \cite[(4.7)]{vis} or \S \ref{pre}. Moreover, the resolvent of $P$
\begin{equation}
    R(s,\zeta):=(P(s)-\zeta)^{-1}: H_{\Lambda} \to H_{\Lambda}
\end{equation}
is a meromorphic family of operators in $\zeta$ for $\zeta\in (-\zeta_0, \zeta_0)+i(-\zeta_0\theta, \infty)$.
The poles of $R(s,\zeta)$ are then called the resonances of $P(s)$ in $(-\zeta_0, \zeta_0)+i(-\zeta_0\theta, \infty)$.

In this paper, we consider $P(s)$ as perturbations of $P=P(0)$. The eigenvalues of $P$ can be approximated by the resonances of $P(s)$ as $s$ goes to $0$. Moreover, we have the following

\begin{theo}
\label{theorem1}
Suppose $P(s)\in \Psi^0(\mathbb{T}^n)$ is a family of 0th order pseudodifferential operators satisfying \eqref{assumption1}, \eqref{assumption2}, \eqref{assumption3} and \eqref{assumption4}. Suppose $\lambda\in\RR$ is an embedded eigenvalue of $P=P(0)$ with multiplicity $m\geq 1$. Then there exists $\lambda_{\ell}(s)$, $1\leq \ell\leq m$ such that $\lambda_{\ell}(s)$ are resonances of $P(s)$, $\lambda_{\ell}(0)=\lambda$ and
\begin{enumerate}
    \item (Fermi golden rule) If $m=1$, then $\lambda_1$ is analytic in $(-s_1,s_1)$, $0<s_1<s_0$ and 
    \begin{equation}
        \Im\ddot{\lambda}_1=-2\Im\langle \Pi_{\perp}R(\lambda)\Pi_{\perp}\dot{P}u, \dot{P}u \rangle_{L^2{(\mathbb{T}^n)}}.
    \end{equation}
    Here $u\in H_{\Lambda}\cap L^2(\mathbb{T}^n)$ is the $L^2$ normalized eigenfunction of $P$ and $\Pi_{\perp}$ is the $L^2$ orthogonal projection onto the orthogonal complement of the eigenspace with eigenvalue $\lambda$;
    
    \item If $m>1$, then $\lambda_{\ell}$ can be labeled so that $\lambda_{\ell}$'s have Puiseux series expansions in $s$. If 
    \begin{equation}
        \lambda_{\ell}(s)=\lambda+c_1 \omega^{\ell} s^{1/p}+c_2\omega^{2\ell}s^{2/p}+\cdots, \quad 1\leq \ell\leq p,
    \end{equation}
    is a Puiseux cycle with $\omega=e^{2\pi i/p}$. Then either $p=1$ and $c_j\in \RR$, $j\in \NN$, or
    \begin{equation}
    \label{mnr}
        \lambda_{\ell}(s)=\lambda+c_ps+\cdots+c_{2Jp}s^{2J}+c_{2Jp+1}\omega^{\ell}s^{2J+1/p}+\cdots
    \end{equation}
    with $c_p, \cdots, c_{2(J-1)p}\in \RR$ and $\Im c_{2Jp}<0$.
\end{enumerate}
\end{theo}
\Remark
As one can see from the proof, in the case where $\lambda$ is a simple eigenvalue, we do not need to assume $p(s)$ is analytic in $s$, we only need $p(s)$ is smooth in $s$ and then in the theorem, $\lambda_1(s)$ is smooth in $s$ and we still have the Fermi golden rule.

Now we state the theorem for viscosity limits.
\begin{figure}[h]
    \centering
    \subfigure[Eigenvalues $\lambda(t)$ of $P(t)$ near 0]{
    \includegraphics[scale=0.3]{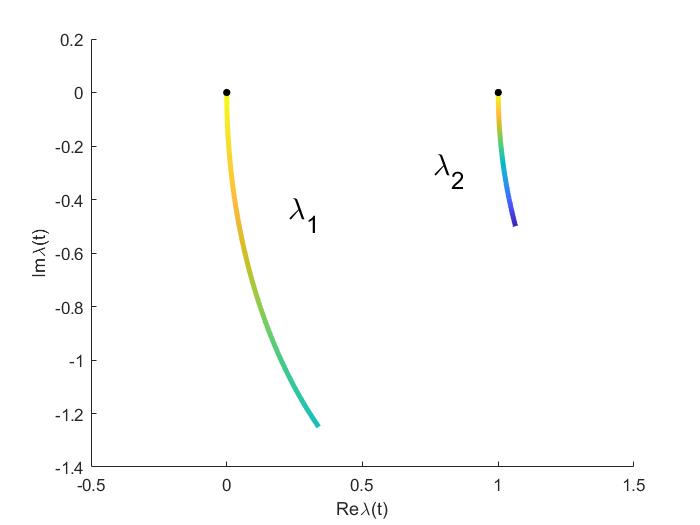}}
    \subfigure[$\Im\lambda(t)$]{
    \label{vis_im}
    \includegraphics[scale=0.3]{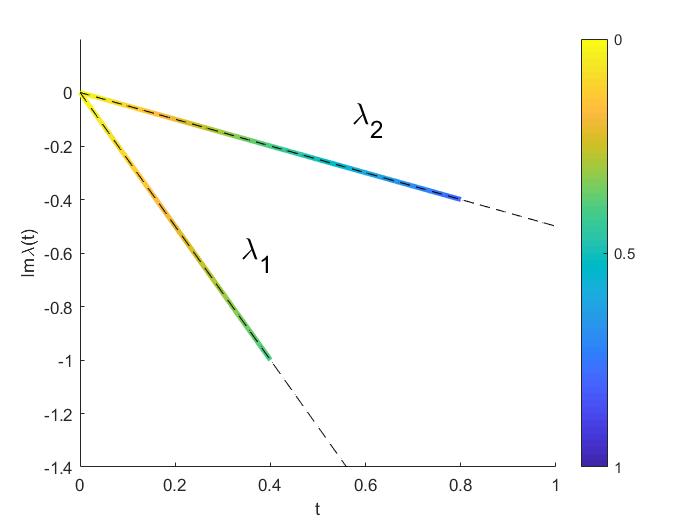}}
    \caption{Eigenvalues of $P(t)=P+it\Delta$, where $P$ is of the form \eqref{pex} with $V_m$, $V_a$ given by \eqref{visv}. In this example, $P$ has eigenvalues $\lambda_1=0$, $\lambda_2=1$ with eigenfunctions $u_1=\tfrac{1}{2\sqrt{2}\pi}\left(ie^{-ix_1}+e^{-2ix_1}\right)$, $u_2=\tfrac{1}{2\sqrt{2}\pi}\left(i+e^{ix_1}\right)$. As computed in \eqref{visex}, $\dot{\lambda}_1=-5i/2$, $\dot{\lambda}_2=-i/2$.}
    \label{vis}
\end{figure}
\begin{theo}
\label{theorem3}
Suppose $P\in \Psi^0(\mathbb{T}^n)$ is a 0th order pseudodifferential operator satisfying \eqref{assumption1}, \eqref{assumption3}, \eqref{assumption4} without dependence on $s$. Let $P(t)=P+it\Delta_{\mathbb{T}^n}$. If $\lambda\in \RR$ is an eigenvalue of $P$ with $L^2$ normalized eigenfunction $u\in H_{\Lambda}\cap L^2(\mathbb{T}^n)$. Then there exists $\lambda(t)\in C^{\infty}((0, t_0);\CC)$, $t_0>0$ such that $\lambda(t)$ is an $L^2$ eigenvalue of $P(t)$, $\lambda(t)\to \lambda_0$ as $t\to 0^+$ and 
\begin{equation}
    \dot{\lambda}=-i\|\nabla u\|_{L^2(\mathbb{T}^n)}^2
\end{equation}
where $\dot{\lambda}=\dot{\lambda}(0)$.
\end{theo}

\subsection{Examples}
As in \cite[Example 1]{circle}, \cite[(B.4)]{vis}, we consider operators of the following form on $\mathbb{T}^2$:
\begin{equation}
\label{pex}
    P=\langle D\rangle^{-1}D_{x_2}+\sin(x_1)(I-\mathbf{V_m}(D_{x_1}))+(I-\mathbf{V_m}(D_{x_1}))\sin(x_1)+\mathbf{V_a}(D_{x_1})
\end{equation}
with $\mathbf{V_m}$, $\mathbf{V_a}$ satisfy
\begin{equation}
    \max\left\{ |\mathbf{V_m}(\xi_1)|, |\mathbf{V_a}(\xi_1)| \right\}\leq Ce^{-c|\Re \xi_1|^2}, \quad |\Im\xi_1|<b\langle \Re\xi_1 \rangle.
\end{equation}
Let $n\in \mathbb{Z}^2$ be fixed, then
\begin{equation}\begin{split}
\label{mode}
    P(e^{in\cdot x})
    = & \left( \langle n \rangle^{-1}n_2+\mathbf{V_a}(n_1) \right)e^{in\cdot x}+(1-\mathbf{V_m}(n_1))\sin(x_1)e^{in\cdot x} \\
    & + \tfrac{1-\mathbf{V_m}(n_1+1)}{2i}e^{i(n+e_1)\cdot x}-\tfrac{1-\mathbf{V_m}(n_1-1)}{2i}e^{i(n-e_1)\cdot x}\\
    = & \left( \langle n \rangle^{-1}n_2+\mathbf{V_a}(n_1) \right)e^{in\cdot x}+\tfrac{2-\mathbf{V_m}(n_1)-\mathbf{V_m}(n_1+1)}{2i}e^{i(n+e_1)\cdot x}\\ & -\tfrac{2-\mathbf{V_m}(n_1)-\mathbf{V_m}(n_1-1)}{2i}e^{i(n-e_1)\cdot x}
\end{split}\end{equation}
with $e_1=(1,0)\in \mathbb{Z}^2$.
Suppose $\mathbf{V_m}(s,\xi_1)$ is an analytic perturbation of $\mathbf{V_m}$ and $P(s)$ is the perturbation of $P$ with $\mathbf{V_m}$ replaced by $\mathbf{V_m}(s,\cdot)$, then we have
\begin{equation}\begin{split}
\label{modep}
    \dot{P}(e^{in\cdot x})
    = & -\left(\sin(x_1) \mathbf{\dot V_m}(D_{x_1})+\mathbf{\dot{V}_m}(D_{x_1})\sin(x_1)\right)(e^{in\cdot x})\\
    = & \tfrac{\mathbf{\dot{V}_m}(n_1)+\mathbf{\dot{V}_m}(n_1-1)}{2i}e^{i(n-e_1)\cdot x} 
    -\tfrac{\mathbf{\dot{V}_m}(n_1)+\mathbf{\dot{V}_m}(n_1+1)}{2i}e^{i(n+e_1)\cdot x}.
\end{split}\end{equation}

\noindent
\textbf{Example 1} (Simple eigenvalues). We first put
\begin{equation}
\label{simplev}
    \mathbf{V_m}(\xi_1)=2\sin^2\left(\tfrac{\pi\xi_1}{2}\right)e^{-(\xi_1-1)^2}, \quad \mathbf{V_a}(\xi_1)=5(\xi_1-1)e^{-\xi_1^2}.
\end{equation}
By \eqref{mode}, $0$ is a simple eigenvalue of $P$ with eigenfunction $u(x)=e^{ix_1}/2\pi$. We consider the following perturbation of $\mathbf{V_m}$:
\begin{equation}
\label{simplevv}
    \mathbf{V_m}(s,\xi_1)=\mathbf{V_m}(\xi_1)+se^{-(\xi_1-1)^2}.
\end{equation}
Figure \ref{simple} shows a numerical illustration of the resonances of $P(s)$ near the eigenvalue $0$. The numerical computation is based on the method of complex scaling, see \cite[Appendix B]{vis}.
To compute $\ddot{\lambda}$ in Theorem \eqref{theorem1}, we notice that by \eqref{modep}, we have
\begin{equation}
    \dot{P}(u)=\tfrac{1+e^{-1}}{4\pi i}\left(1-e^{2ix_1}\right).
\end{equation}
Now we find approximate values of $a_n\in \CC$ satisfying
\begin{equation*}
    R(0)\dot P u=\sum a_ne^{in\cdot x}.
\end{equation*}
Since $\dot Pu$ is a function that is independent of $x_2$ and the only dependence of $P$ on $x_2$ is the term $\langle D \rangle^{-1}D_{x_2}$, we can ignore $x_2$ and replace $x_1$ by $x$ in the following computation.
For $\epsilon>0$, we now find $v(\epsilon)=\sum_{n\in\mathbb{Z}}a_n(\epsilon) e^{inx}\in L^2$ such that $(P-i\epsilon)v=\dot P u$. Notice that
\begin{equation*}
    (P-i\epsilon)v=\sum \left( \tfrac{2-\mathbf{V_m}(n-1)-\mathbf{V_m}(n)}{2i}a_{n-1}+(\mathbf{V_a}(n)-i\epsilon)a_n-\tfrac{2-\mathbf{V_m}(n)-\mathbf{V_m}(n+1)}{2i}a_{n+1}\right)e^{inx}.
\end{equation*}
For fixed $N>0$, we consider $a_n=a_n(\epsilon, N)$ satisfying the following equations:
$$
    a_n=\left( \tfrac{\epsilon-\sqrt{\epsilon^2+4}}{2} \right)^{n-N}a_N, \  n\geq N, \ a_n=\left( \tfrac{\epsilon+\sqrt{\epsilon^2+4}}{2} \right)^{n+N}a_{-N}, \ n\leq -N. 
$$
$$\tfrac{2-\mathbf{V_m}(n-1)-\mathbf{V_m}(n)}{2i}a_{n-1}+(\mathbf{V_a}(n)-i\epsilon)a_n-\tfrac{2-\mathbf{V_m}(n)-\mathbf{V_m}(n+1)}{2i}a_{n+1}=k_n, \ -N\leq n\leq N, $$
$$k_n=0, \ -N\leq n\leq N, N\neq 0,2, \ \ k_0=\tfrac{1+e^{-1}}{4\pi i}, \ k_2=-\tfrac{1+e^{-1}}{4\pi i}.$$
Let $w(\epsilon, N):=\sum a_n(\epsilon, N)e^{inx}$, then $w\in L^2$ and
\begin{equation*}\begin{split}
    & r(\epsilon,N):=(P-i\epsilon)w(\epsilon, N)-\dot Pu \\
    = & \sum_{|n|\geq N+1}\left( -\tfrac{\mathbf{V_m}(n-1)+\mathbf{V_m}(n)}{2i}a_{n-1}+\mathbf{V_a}(n)a_n+\tfrac{\mathbf{V_m}(n)+\mathbf{V_m}(n+1)}{2i}a_{n+1} \right)e^{inx}.
\end{split}\end{equation*}
For $n=0,2$, we have
\begin{equation*}\begin{split}
    & |a_n(\epsilon, N)-a_n|
    \leq |a_n(\epsilon,N)-a_n(\epsilon)|+|a_n(\epsilon)-a_n|\\
    \leq & C\|(P-i\epsilon)^{-1}r(\epsilon,N)\|_{H^{-1/2-}}+C\|(P-i\epsilon)^{-1}\dot Pu-R(0)\dot Pu\|_{H^{-1/2-}}.\\
\end{split}\end{equation*}
Note that
\begin{equation*}
    (P-i\epsilon)^{-1}r(\epsilon,N)=\left((P+I)\Pi_{\perp}-I-i\epsilon\right)^{-1}r(\epsilon,N),
\end{equation*}
Since $0$ is not an eigenvalues of $(P+I)\Pi_{\perp}-I$, by the proof of \cite[Proposition 3.1]{scattering}, we have
\begin{equation*}
    \|\left((P+I)\Pi_{\perp}-I-i\epsilon\right)^{-1}r(\epsilon,N)\|_{H^{-1/2-}}\leq C\|r(\epsilon,N)\|_{H^{1/2+}}\to 0, \ \ \epsilon\to 0, \ N\to\infty.
\end{equation*}
By the same proposition, see also \cite[Lemma 3.3]{force}, we also have
\begin{equation*}
    \|(P-i\epsilon)^{-1}\dot Pu-R(0)\dot Pu\|_{H^{-1/2-}}\to 0, \  \epsilon \to 0.
\end{equation*}
Thus we find
\begin{equation}
    \lim_{\epsilon\to 0, \ N\to \infty}\left( |a_0(\epsilon,N)-a_0|+|a_2(\epsilon,N)-a_2| \right)=0.
\end{equation}
To get an approximate value of $a_0$, $a_2$ in our example, we put $N=3$ and set $\epsilon\to 0$ and find $a_0\approx 0.0003+0.0230i$, $a_2\approx -0.1100+0.0101i$. As a result we find
\begin{equation*}
    \Im\ddot{\lambda}=-2\Im\langle \Pi_{\perp}R(0)\Pi_{\perp}\dot Pu, \dot Pu \rangle_{L^2}=2\pi(1+e^{-1})\Re(a_2-a_0)\approx -0.9479.
\end{equation*}
Therefore
\begin{equation}
\label{simpleimap}
    \Im\lambda = \tfrac12 \left(\Im\ddot{\lambda}\right) s^2+O(s^3)\approx -0.4739 s^2+O(s^3).
\end{equation}
The results are shown in Figure \ref{simple}.

\noindent
\textbf{Example 2} (Eigenvalues with multiplicities). Now we consider
\begin{equation}
\label{multiv}
    \mathbf{V_m}(\xi_1)=2\sin^2(\tfrac{\pi\xi_1}{2})e^{-(\xi_1^2-1)^2}, \quad \mathbf{V_a}(\xi_1)=6(1-\xi_1^2)e^{-(\xi_1-1)^2}.
\end{equation}
In this case, by \eqref{mode}, $\lambda=0$ is an eigenvalue of $P$ with eigenfunction $u_1=e^{-ix}/2\pi$, $u_2=e^{ix}/2\pi$.
Let $\mathbf{V_m}(s,\xi_1)$ be a perturbation of $P$ as follows:
\begin{equation}
\label{multivv}
    \mathbf{V_m}(s,\xi_1)=\mathbf{V_m}(\xi_1)+s\left( e^{-(\xi_1+4/5)^2}+\tfrac45 e^{-(\xi_1-4/5)^2} \right).
\end{equation}
Figure \ref{multi} shows a numerical result of the resonances of $P(s)$ near $0$. As predicted in Theorem \ref{theorem1}, the resonances of $P(s)$ branches into Puiseux series of cycle $p=2$.
In fact, by the proof of Theorem \ref{theorem1} and \eqref{fpm}, \eqref{epm}, \eqref{ginv}, we have
\begin{equation}
\begin{split}
    F_{-+}
    =  E_{-+}+\sum_{k=1}^{\infty}(-1)^{k}s^kE_-\dot{P}(E\dot{P})^{k-1}E_+
    =  \lambda I_2-sA(s)
\end{split}
\end{equation}
with
\begin{equation}
    A(s)=\sum_{k=0}^{\infty}(-1)^ks^k\begin{pmatrix} \langle \dot{P}(E\dot{P})^k u_1,u_1 \rangle & \langle \dot{P}(E\dot{P})^k u_1,u_2 \rangle \\ \langle \dot{P}(E\dot{P})^k u_2,u_1 \rangle & \langle \dot{P}(E\dot{P})^ku_2,u_2 \rangle \end{pmatrix}=:\begin{pmatrix} A_{11} & A_{12} \\ A_{21} & A_{22} \end{pmatrix}.
\end{equation}
Thus
\begin{equation}
    \textrm{det}(F_{-+})=\lambda^2-(A_{11}+A_{22})s\lambda+\textrm{det}(A)s^2.
\end{equation}
Put $\textrm{det}(F_{-+})=0$, use the fact that $A(0)=0$ in this example, and we find,
\begin{equation}
    \lambda_{1,2}= \tfrac{\dot{A}_{11}+\dot{A}_{22}\pm\sqrt{(\dot{A}_{11}-\dot{A}_{22})^2+4\dot{A}_{12}\dot{A}_{21}}}{2}s^2+O(s^{5/2}), \  \dot{A}_{ij}=-\langle \dot{P}E\dot{P}u_i, u_j \rangle.
\end{equation}
By \eqref{modep}, we have
\begin{equation*}
    \dot{P}(e^{-ix_1})=\tfrac{\mathbf{\dot{V}_m}(-2)+\mathbf{\dot{V}_m}(-1)}{2i}e^{-2ix}-\tfrac{\mathbf{\dot{V}_m}(-1)+\mathbf{\dot{V}_m}(0)}{2i}, \ \ \dot{P}(e^{ix_1})=\tfrac{\mathbf{\dot{V}_m}(0)+\mathbf{\dot{V}_m}(1)}{2i}-\tfrac{\mathbf{\dot{V}_m}(1)+\mathbf{\dot{V}_m}(2)}{2i}e^{2ix}.
\end{equation*}
Now we find $a_n$, $b_n$ such that
\begin{equation*}
    R(0)P(e^{-ix})=\sum a_ne^{inx}, \ \  R(0)P(e^{ix})=\sum b_ne^{inx}.
\end{equation*}
By the same analysis in Example 1, we solve 
$$a_{-3}=\left(\tfrac{\epsilon+\sqrt{\epsilon^2+4}}{2}\right)a_{-2}, \ \ b_{3}=\left( \tfrac{\epsilon-\sqrt{\epsilon^2+4}}{2} \right)b_2, \ \ a_2=0, \ \ b_{-2}=0,$$
$$ \tfrac{2-\mathbf{V_m}(-3)-\mathbf{V_m}(-2)}{2i}a_{-3}+(\mathbf{V_a}(-2)-i\epsilon)a_{-2}-\tfrac{2-\mathbf{V_m}(-2)-\mathbf{V_m}(-1)}{2i}a_{-1}=\tfrac{\mathbf{\dot{V}_m}(-2)+\mathbf{\dot{V}_m}(-1)}{2i}, $$
$$ \tfrac{2-\mathbf{V_m}(-1)-\mathbf{V_m}(0)}{2i}a_{-1}+(\mathbf{V_a}(0)-i\epsilon)a_{0}-\tfrac{2-\mathbf{V_m}(0)-\mathbf{V_m}(1)}{2i}a_{1}= -\tfrac{\mathbf{\dot{V}_m}(-1)+\mathbf{\dot{V}_m}(0)}{2i},$$
$$ \tfrac{2-\mathbf{V_m}(-1)-\mathbf{V_m}(0)}{2i}b_{-1}+(\mathbf{V_a}(0)-i\epsilon)b_{0}-\tfrac{2-\mathbf{V_m}(0)-\mathbf{V_m}(1)}{2i}b_{1}= \tfrac{\mathbf{\dot{V}_m}(0)+\mathbf{\dot{V}_m}(1)}{2i},$$
$$ \tfrac{2-\mathbf{V_m}(1)-\mathbf{V_m}(2)}{2i}b_1+(\mathbf{V_a}(2)-i\epsilon)b_{2}-\tfrac{2-\mathbf{V_m}(2)-\mathbf{V_m}(3)}{2i}b_{3}=-\tfrac{\mathbf{\dot{V}_m}(1)+\mathbf{\dot{V}_m}(2)}{2i} $$
and let $\epsilon\to 0$ and find approximately
\begin{equation*}
    a_{-2}\approx 0.6147+0.0014i, \quad a_0\approx 0.4397i, \quad b_0\approx 0.3980i, \quad b_2\approx-0.0111-0.0737i.
\end{equation*}
Now we have
\begin{equation*}\begin{split}
    & \dot{A}_{11}=\tfrac{\mathbf{\dot V_m}(-2)+\mathbf{\dot V_m}(-1)}{2i}a_{-2}-\tfrac{\mathbf{\dot V_m}(-1)+\mathbf{\dot V_m}(0)}{2i}a_0, \quad \dot A_{12}=\tfrac{\mathbf{\dot V_m}(0)+\mathbf{\dot V_m}(1)}{2i}a_0,\\
    & \dot A_{21}=-\tfrac{\mathbf{\dot V_m}(-1)+\mathbf{\dot V_m}(0)}{2i}b_0, \quad \dot A_{22}=\tfrac{\mathbf{\dot V_m}(0)+\mathbf{\dot V_m}(1)}{2i}b_0-\tfrac{\mathbf{\dot V_m}(1)+\mathbf{\dot V_m}(2)}{2i}b_2.
\end{split}\end{equation*}
Inserting the values of $a_{-2}$, $a_0$, $b_0$, $b_2$ and we find
\begin{equation*}
\dot A\approx\begin{pmatrix} -0.4260-0.3778i & 0.3863\\0.3863 & -0.3129-0.0055i \end{pmatrix}.
\end{equation*}
Therefore
\begin{equation}
\label{multiimap}
    \Im\lambda_1\approx -0.1611s^2+O(s^{5/2}), \quad \Im\lambda_2\approx -0.2222s^2+O(s^{5/2}).
\end{equation}
The results are shown in Figure \ref{multi}.

\noindent
\textbf{Example 3} (Simple eigenvalues and operators with viscosity). Let $P$ be as in Example 1, that is, $\mathbf{V_m}$, $\mathbf{V_a}$ are given by $\eqref{simplev}$. If we add the viscosity to $P$ and consider $P(t)=P+it\Delta_{\mathbb{T}^2}$, then
\begin{equation}
    P(t)(e^{ix_1})=it\Delta_{\mathbb{T}^2}(e^{ix_1})=-ite^{ix_1}.
\end{equation}
Therefore $\lambda(t)=-it$ is the eigenvalue of $P(t)$ near $0$. Hence $\dot{\lambda}=-i$. This verifies the formula in Theorem \ref{theorem3}. As a less trivial example, we consider
\begin{equation}\begin{split}
\label{visv}
    \mathbf{V_m}(\xi_1) & =(\xi_1+3)(\xi_1-2)(\tfrac12\xi_1^3+\tfrac{7}{12}\xi_1^2-\tfrac{11}{12}\xi_1-\tfrac{2}{3})e^{-\xi_1(\xi_1+3)(\xi_1^2-1)(\xi_1^2-4)},\\
    \mathbf{V_a}(\xi_1) & =(-\tfrac{2}{3}\xi_1^3-\xi_1^2+\tfrac{5}{3}\xi_1+3)e^{-\xi_1(\xi_1^2-1)(\xi_1+2)}.
\end{split}\end{equation}
Then $P$ has eigenvalues $\lambda_1=0$, $\lambda_2=1$ with eigenfunctions $u_1=\tfrac{1}{2\sqrt{2}\pi}\left(ie^{-ix_1}+e^{-2ix_1}\right)$, $u_2=\tfrac{1}{2\sqrt{2}\pi}\left(i+e^{ix_1}\right)$. Thus
\begin{equation}
\label{visex}
    \dot{\lambda}_1=-i\|\nabla u_1\|_{L^2(\mathbb{T}^2)}^2=-\tfrac{5i}{2}, \quad \dot{\lambda}_2=-i\|\nabla u_2\|_{L^2(\mathbb{T}^2)}^2=-\tfrac{i}{2}.
\end{equation}
Figure \ref{vis} shows the numerical results of the eigenvalues of $P(t)$ near $0$. Figure \ref{vis_im} justifies \eqref{visex}.

\subsection{Organization of the paper}
In \S \ref{pre}, we review the construction of the space of hyperfunctions and Grushin problem briefly. In \S \ref{anal}, we show the analyticity of the eigenfunctions of $P$. In \S \ref{pf1}, we give a proof to Theorem \ref{theorem1}. In \S \ref{pf2}, we prove Theorem \ref{theorem3}.

\medskip\noindent\textbf{Acknowledgements.}
I would like to thank Maciej Zworski for suggesting this problem, for helpful advice and for his help in Matlab experiments. Partial support by the National Science Fundation grant DMS-1952939 is also gratefully acknowledged.

\section{preliminaries}
\label{pre}

\subsection{Space of hyperfunctions}
\label{hspace}
We first review the function spaces $H_{\Lambda}$ constructed in \cite[(4.7)]{vis}.

Let $\widetilde{T^*\mathbb{T}^n}=\{ (z,\zeta): z\in \CC^n/2\pi \mathbb{Z}^n, \zeta\in \CC^n \}$. Let $\sigma$ be the complex symplectic form over $\widetilde{T^*\mathbb{T}^n}$.
We assume the function $F$ in \eqref{assumption4} satisfies
\begin{equation}
    \sup_{|\alpha|+|\beta|\leq 2}\langle \xi \rangle^{-1+|\beta|}|\partial_x^{\alpha}\partial_{\xi}^{\beta}F(x,\xi)|\leq \epsilon_0, \quad \sup_{\alpha, \beta}\langle \xi \rangle^{-1+|\beta|}|\partial_x^{\alpha}\partial_{\xi}^{\beta}F(x,\xi)|\leq C_{\alpha, \beta}
\end{equation}
with $\epsilon_0>0$ small enough.
We put
\begin{equation}
    \Lambda:=\{ (x+iF_{\xi}(x,\xi), \xi-iF_{x}(x,\xi)): (x,\xi)\in T^*\mathbb{T}^n \}\subset \widetilde{T^*\mathbb{T}^n},
\end{equation}
Then 
\begin{equation}
    \Im(\sigma|_{\Lambda})\equiv 0, \quad \Re (\sigma|_{\Lambda}) \text{ is non-degenerate.}
\end{equation}
Let $d\alpha$ be the nutural volume form on $\Lambda$ defined by $d\alpha=(\sigma|_{\Lambda})^n/n!$. Let 
\begin{equation}
    H(x,\xi)=F(x,\xi)-\xi\cdot F_{\xi}(x,\xi)\in C^{\infty}(\Lambda; \RR).
\end{equation}
Let $T_{\Lambda}$ be the complex deformation of the FBI transform defined in \cite[\S 4]{vis}.

We now define a function space
\begin{equation}
    \mathscr{I}_{\delta}=\{ u\in L^2(\mathbb{T}^n): \|u\|_{\mathscr{I}_{\delta}}<\infty \},
\end{equation}
where
\begin{equation}
    \|u\|_{\mathscr{I}_{\delta}}^2:=\sum_{n\in \mathbb{Z}^n}|\widehat{u}(n)|^2e^{4|n|\delta}, \quad \widehat{u}(n)=\frac{1}{(2\pi)^n}\int_{\mathbb{T}^n}u(x)e^{-i\langle x,n \rangle}dx.
\end{equation}
Let $\epsilon_0$ be sufficiently small such that $T_{\Lambda}$ has the mapping property as in \cite[Lemma 4.1]{vis} for $0<\delta<\delta_0$. Then the space $H_{\Lambda}$ is defined as the closure of $\mathscr{I}_{\delta}$ with respect to the norm $\|\cdot\|_{H_{\Lambda}}$ given by the formula
\begin{equation}
    \|u\|_{H_{\Lambda}}:=\int_{\Lambda}|T_{\Lambda}u(\alpha)|^2e^{-2H(\alpha)/h}d\alpha.
\end{equation}

\subsection{Grushin problems}
We briefly review Grushin problems, for a complete introduction, see for instance \cite[\S C.1]{res}. See also \cite{elementary} for applications of Grushin problems.

Suppose $P: X_1\to X_2$, $R_-: X_-\to X_2$, $R_+: X_1\to X_+$ are bounded operators on Banach spaces $X_1, X_2, X_-, X_+$. 
We call the equation
\begin{equation}
\label{grushin}
    \begin{pmatrix} P & R_- \\ R_+ & 0 \end{pmatrix}\begin{pmatrix} u \\ u_- \end{pmatrix}=\begin{pmatrix} v\\v_+ \end{pmatrix}
\end{equation}
a Grushin problem. We call the Grushin problem \eqref{grushin} well-posed if it is invertible, and in this case we write the inverse as
\begin{equation}
\label{grushininv}
    \begin{pmatrix} u\\u_- \end{pmatrix}=\begin{pmatrix} E & E_+ \\ E_- & E_{-+} \end{pmatrix}\begin{pmatrix} v\\v_+ \end{pmatrix},
\end{equation}
with operators $E: X_2\to X_1$, $E_+: X_+\to X_1$, $E_-: X_2\to X_-$, $X_+\to X_-$. We also know that $P$ is invertible if and only if $E_{-+}$ is invertible. Moreover, when $P$ and $E_{-+}$ are invertible, we have
\begin{equation}
    P^{-1}=E-E_+E_{-+}^{-1}E_-.
\end{equation}
We also record the following formula for the perturbed Grushin problem
\begin{lemm}(\cite[Lemma C.3]{res})
Suppose the Grushin problem \eqref{grushin} is well-posed and \eqref{grushininv} is the inverse. If $B: X_1\to X_2$ is a bounded operator such that
\begin{equation}
    \|EB\|_{X_1\to X_1}<1, \quad \|BE\|_{X_2\to X_2}<1,
\end{equation}
then the Grushin problem
\begin{equation}
    \begin{pmatrix} P+B & R_- \\ R_+ & 0 \end{pmatrix}: X_1\times X_-\to X_2\times X_+
\end{equation}
is well-posed with inverse
\begin{equation}
    \begin{pmatrix} F& F_+\\ F_-& F_{-+} \end{pmatrix}
\end{equation}
such that
\begin{equation}
\label{fpm}
    F_{-+}=E_{-+}+\sum_{\ell=1}^{\infty}(-1)^{\ell}E_-B(EB)^{\ell-1}E_+.
\end{equation}
\end{lemm}

\section{analyticity of eigenfunctions}
\label{anal}

In this section we prove the analyticity of the eigenfunctions of $P$. More precisely,
\begin{lemm}
\label{aeig}
Suppose $u\in H_{\Lambda}\cap L^2(\mathbb{T}^n)$ is an eigenfunction of $P$ with eigenvalue $\lambda$. Then $u\in \mathscr{I}_{\delta}$ for some $\delta>0$.
\end{lemm}
\begin{proof}
Since $u\in H_{\Lambda}$, by \cite[Proposition 2.5]{vis}, there exists $\psi\in S^1(T^*\mathbb{T}^n)$ such that
\begin{equation}
    \psi(x,\xi)=F(x,\xi)+O(\epsilon_0^2)_{S^1(T^*\mathbb{T}^n)}
\end{equation}
and 
\begin{equation}
    T_{\Lambda}u\in L^2(T^*\mathbb{T}^n, e^{2\psi/h}dxd\xi).
\end{equation}
Now by \cite[Proposition 2.3]{vis} and the fact that $u\in L^2$, we find
\begin{equation}
    \WF_a(u)\cap (U\times \Gamma)=\emptyset
\end{equation}
as long as $|\psi(x,\xi)|\geq |\xi|/C$ when $(x,\xi)\in U\times \Gamma\subset T^*\mathbb{T}^n$. Therefore
\begin{equation}
    \WF_a(u)\subset \{(x,\xi): p(x,\xi)=0, F(x,\xi)=0\}.
\end{equation}
Here $\WF_a(u)$ is the analytic wavefront set of $u$, see \cite[Definition 6.1]{wfa}.

To show $\WF_a(u)=\emptyset$, it remains to show that the analytic singularities of $u$ propagates along the bicharacteristics, since $H_pF(x,\xi)>0$ when $(x,\xi)\in \{p=0\}\cap\{|\xi|>C\}$. More precisely, we have the following

\noindent
\textbf{Claim}: $(x_0,\xi_0)\in \textrm{MS}(u)$ if and only if there exist $t_0\in \RR$ such that $\gamma:[0,t_0]\to T^*\mathbb{T}^n$, $\gamma(t)=e^{tH_p}(x_0,\xi_0)$ exists and $\gamma(t_0)\in \textrm{MS}(u)$. Here $\textrm{MS}(u)$ is the microlocal support of the distribution $u$, see \cite[Definition 3.2.1]{aprop}.

If the claim is true, then $\WF_a(u)\cap\{p=0\}=\emptyset$ by the the assumption \eqref{assumption4}. By \cite[Theorem 4.2.2]{aprop}, $\textrm{MS}(u)\subset \{p=0\}$. Hence we find $\WF_a(u)=\emptyset$. This implies $u\in \mathscr{I}_{\delta}$ for some $\delta>0$.

We now prove the claim.

\noindent
\textbf{Step 1}. We start by noting that for any real-valued $\varphi\in S^0(T^*\mathbb{T}^n)$, $\kappa>0$ we have 
\begin{equation}
\label{commutator}
    \langle (H_p\varphi) e^{\kappa\varphi/h}T_{\Lambda}u, e^{\kappa\varphi/h}T_{\Lambda}u \rangle_{L^2(T^*\mathbb{T}^n)}=O(\kappa+h/\kappa)\|e^{\kappa\varphi/h}T_{\Lambda}u\|^2_{L^2(T^*\mathbb{T}^n)}.
\end{equation}
In fact, by \cite[Proposition 3.3.1]{aprop}, we have
\begin{equation}
    e^{\kappa \varphi/h}T_{\Lambda}Pv=Qe^{\kappa \varphi/h}T_{\Lambda}v,
\end{equation}
with
\begin{equation}
    Q=e^{\kappa\varphi/h}\Op_h(p(x-\xi^*,x^*))e^{-\kappa\varphi/h}
\end{equation}
where $(x^*,\xi^*)$ are the dual variables of $(x,\xi)$ such that $\Op_h(x^*)=hD_{x}$, $\Op_h(\xi^*)=hD_{\xi}$. The principal symbol of $Q$ is 
\begin{equation}
\label{sq}
    \sigma(Q)(x,\xi,x^*,\xi^*)=q_0(x,\xi,x^*,\xi^*)=p(x-\xi^*-i\kappa\partial_{\xi}\varphi, x^*+i\kappa\partial_{x}\varphi).
\end{equation}
Therefore, by \cite[Theorem 3.5.1]{aprop},
\begin{equation}\begin{split}
    & \langle e^{\kappa \varphi/h}T_{\Lambda}Pv, e^{\kappa \varphi/h}T_{\Lambda}v \rangle_{L^2(T^*\mathbb{T}^n)}
    =  \langle Qe^{\kappa\varphi/h}T_{\Lambda}v, e^{\kappa\varphi/h}T_{\Lambda}v \rangle_{L^2(T^*\mathbb{T}^n)}\\
    = & \langle (\widetilde{q}(x,\xi;h)+R(h))e^{\kappa\varphi/h}T_{\Lambda}v, e^{\kappa\varphi/h}T_{\Lambda}v \rangle_{L^2(T^*\mathbb{T}^n)},
\end{split}\end{equation}
with $R(h)  =O(h^{\infty})_{L^2(T^*\mathbb{T}^n)\to L^2(T^*\mathbb{T}^n)}$ and
\begin{equation}
    \begin{split}
        \widetilde{q}  \sim \sum_{j=0}^{\infty}h^j\widetilde{q}_j \text{~~~~in~~~~} S^0(T^*\mathbb{T}^n), \quad  \widetilde{q}_0(x,\xi)  =q_0(x,\xi,\xi-\kappa\partial_{\xi}\varphi(x,\xi), \kappa\partial_x\varphi(x,\xi)).
    \end{split}
\end{equation}
By \eqref{sq}, we have
\begin{equation}
    \begin{split}
        \widetilde{q}_0(x,\xi) & = p(x-\kappa\partial_{x}\varphi-i\kappa\partial_{\xi}\varphi, \xi-\kappa\partial_{\xi}\varphi+i\kappa\partial_x\varphi)\\
        & = p(x,\xi)+\left((\partial_x p\partial_x\varphi-\partial_{\xi}p \partial_{\xi}\varphi)-iH_p\varphi\right)\kappa+O(\kappa^2)_{S^0(T^*\mathbb{T}^n)}.
    \end{split}
\end{equation}
Therefore
\begin{equation}\begin{split}
    & \Im\langle e^{\kappa \varphi/h}T_{\Lambda}Pv, e^{\kappa \varphi/h}T_{\Lambda}v \rangle_{L^2(T^*\mathbb{T}^n)}\\
    = & \langle (\Im\widetilde{q})e^{\kappa\varphi/h}T_{\Lambda}v, e^{\kappa\varphi/h}T_{\Lambda}v \rangle_{L^2(T^*\mathbb{T}^n)}+O(h^{\infty})\|e^{\kappa\varphi/h}T_{\Lambda}v\|_{L^2(T^*\mathbb{T}^n)}\\
    = & -\kappa\langle (H_p\varphi) e^{\kappa\varphi/h}T_{\Lambda}v, e^{\kappa\varphi/h}T_{\Lambda}v \rangle_{L^2(T^*\mathbb{T}^n)} + O(h+\kappa^2)\|e^{\kappa\varphi/h}T_{\Lambda}v\|^2_{L^2(T^*\mathbb{T}^n)}.
\end{split}\end{equation}
Since $Pu=0$, we have \eqref{commutator}.

\noindent
\textbf{Step 2}. We now construct the function $\varphi\in S^0(T^*\mathbb{T}^n)$ to be used in the next part to derive propagation estimates. Such functions are called ``escape functions'', we refer to \cite[Lemma E.48]{res} for the standard construction in the case of principle type propagation.

Let $\Sigma\subset T^*\mathbb{T}^n$ be a hypersurface that passes $\gamma(0)$ and is a cross section of the flow $e^{tH_p}$ such that
\begin{equation}
    \Phi: (-2\delta, t_1+\delta)\times \Sigma\to T^*\mathbb{T}^n, \quad \Phi(t,(y,\eta))=e^{tH_p}(y,\eta),
\end{equation}
is a diffeomorphism to its image $V:=\Phi((-2\delta,t_1+\delta)\times\Sigma)$.
Since $\gamma(t_1)\notin \textrm{MS}(u)$, we assume 
\begin{equation}
\label{expd}
    \|T_{\Lambda}\|_{L^2(W)}=O(e^{-\alpha/h}), \quad  W:= \Phi((t_1-\delta, t_1+\delta)\times \Sigma_1),
\end{equation}
with $\delta>0$ small enough, $\Sigma_1$ an open neighborhood of $\gamma(0)$, $\overline{\Sigma_1}\subset \Sigma$. Let $\Sigma_1^{\prime}$ be an open neighborhood of $\Sigma_1$ such that $\overline{\Sigma_1^{\prime}}\subset \Sigma_1$.
Let $\chi\in C_c^{\infty}(\Sigma, [0,1])$, $\rho\in C_c^{\infty}(\RR; [0,\alpha])$, such that
\begin{equation}
    \chi |_{\Sigma_1^{\prime}}\equiv 1, \quad \chi |_{\Sigma_1\setminus\Sigma_1^{\prime}}\neq 0, \quad \chi |_{\Sigma\setminus \Sigma_1}\leq 1/4,
\end{equation}
and
\begin{equation}
    \supp\rho\subset (-2\delta, t_1+\delta), \quad \rho(0)=\alpha/2, \quad \rho^{\prime}|_{(-2\delta, -\delta)}\geq 0, \quad \rho^{\prime}|_{(-\delta, t_1-\delta)} \geq \beta
\end{equation}
with $\beta>0$.
Now we put
\begin{equation}
    \varphi(\Phi(t,(y,\eta))):=\chi(y,\eta)\rho(t), \quad (t,(y,\eta))\in (-2\delta, t_1+\delta)\times \Sigma
\end{equation}
and extend $\varphi$ by $0$ outside $V$. Now we have $\varphi\in S^0(T^*\mathbb{T}^n)$ which satisfies
\begin{equation}
    0\leq \varphi\leq \alpha, \quad \supp\varphi\subset V, \quad H_p\varphi|_{\Phi(t,(y,\eta))}=\chi(y,\eta)\rho^{\prime}(t).
\end{equation}

\noindent
\textbf{Step 3}.
Let $V_1:=\Phi((-\delta,t_1-\delta)\times \Sigma_1)$, then 
\begin{equation}
    H_p\varphi|_{V_1} \geq 1/C_1>0.
\end{equation}
By \eqref{commutator}, we find
\begin{equation*}
    \begin{split}
         \|e^{\kappa\varphi/h}T_{\Lambda}u\|^2_{L^2(V_1)}
        \leq & C_1\|e^{\kappa\varphi/h}T_{\Lambda}u\|^2_{L^2(T^*\mathbb{T}^n\setminus V_1)}
         +CC_1(\kappa+h/\kappa)\|e^{\kappa\varphi/h}T_{\Lambda}u\|^2_{L^2(T^*\mathbb{T}^n)}.
    \end{split}
\end{equation*}
Let $\kappa=1/(1+2CC_1)$, then for $0<h<h_0$ where $h_0>0$ is sufficiently small we have
\begin{equation}
    \|e^{\kappa\varphi/h}T_{\Lambda}u\|_{L^2(V_1)}\leq C(h_0)\|e^{\kappa\varphi/h}T_{\Lambda}u\|_{L^2(T^*\mathbb{T}^n\setminus V_1)}
\end{equation}
with $C(h_0)>0$.
Now we denote
\begin{equation}\begin{split}
    & D_1:=T^*\mathbb{T}^n\setminus V, \quad D_2:=W,\\
    & D_3:=\Phi((-\delta,t_1+\delta)\times (\Sigma\setminus \Sigma_1)), \quad D_4:=\Phi((-2\delta,-\delta)\times \Sigma).
\end{split}\end{equation}
Then $T^*\mathbb{T}^n\setminus V_1=\sqcup_{j=1}^4 D_j$
and
\begin{equation}
    \varphi|_{D_1}=0, \quad \varphi|_{D_2}\leq \alpha, \quad \varphi|_{D_3}\leq \alpha/4, \quad \varphi|_{D_4}\leq \alpha/2-\beta\delta.
\end{equation}
Let $\beta<\alpha/(4\delta)$, then we have
\begin{equation}\begin{split}
    \|e^{\kappa\varphi/h}T_{\Lambda}u\|_{L^2(V_1)}
    \leq & C\sum_{j=1}^4\|e^{\kappa\varphi/h}T_{\Lambda}u\|_{L^2(D_j)}\\
    = & O(1+e^{\max\{ {(\kappa-1)\alpha/h}, {\kappa\varphi/(4h)}, \kappa(\alpha/2-\beta\delta)/h\}} )
    =  O(1+e^{\kappa(\alpha/2-\beta\delta)}).
\end{split}\end{equation}
Here we used \eqref{expd}.
Now let $V_1^{\prime}$ be an open neighborhood of $\gamma(0)$ such that
\begin{equation}
    \overline{V_1^{\prime}}\subset V_1, \quad \varphi|_{V_1^{\prime}}\geq \alpha/2-\beta\delta/2.
\end{equation}
Then we have
\begin{equation}
    \|T_{\Lambda}u\|_{L^2(V_1^{\prime})}=O(e^{-\kappa\beta\delta/h}).
\end{equation}
This implies $\gamma(0)\notin \textrm{MS}(u)$.
\end{proof}

\section{Proof of Theorem \ref{theorem1}}
\label{pf1}
We now give a proof to Theorem \ref{theorem1}.

\begin{proof}[Proof of Theorem \ref{theorem1}]
\textbf{Part 1}.
Let $\{u_1, \cdots, u_m\}\subset H_{\Lambda}\cap L^2(\mathbb{T}^n)$ be an orthonormal basis of the eigenspace with eigenvalue $\lambda$.
We consider the Grushin problem
\begin{equation}
    \mathscr{P}(s,\zeta)=\begin{pmatrix} P(s)-\zeta & R_- \\ R_+ & 0 \end{pmatrix}: H_{\Lambda}\times \CC^m \to H_{\Lambda}\times \CC^m,
\end{equation}
where $R_-: \CC^m\to H_{\Lambda}$ and $R_+: H_{\Lambda}\to \CC^m$ are defined by
\begin{equation}
    R_+w=(\langle w,u_1 \rangle_{L^2(\mathbb{T}^2)}, \cdots, \langle w,u_m \rangle_{L^2(\mathbb{T}^2)}), \quad R_-(w_-^1,\cdots, w_-^m)=\sum_{j=1}^m w_-^ju_{j}.
\end{equation}
By Lemma \ref{aeig}, $u_{j}\in \mathscr{I}_{\delta}$ for small $\delta>0$. Therefore
\begin{equation}
    |\langle u,u_{j} \rangle_{L^2(\mathbb{T}^n)}|\leq \|u_{j}\|_{\mathscr{I}_{\delta}}\|u\|_{\mathscr{I}_{-\delta}}\leq C \|u_{j}\|_{\mathscr{I}_{\delta}}\|u\|_{H_{\Lambda}}
\end{equation}
since $\mathscr{I}_{-\delta}$ is the dual space of $\mathscr{I}_{\delta}$ relative $L^2$ pairing on $\mathbb{T}^n$ and $\mathscr{I}_{\delta}\subset H_{\Lambda}\subset \mathscr{I}_{-\delta}$ (see for instance \cite[\S 4]{vis}). This implies $R_{\pm}$ are well-defined and bounded operators.

When $s=0$ and $\zeta$ is near $\lambda$, by \cite[Lemma 7.9]{vis}, there exist $A(\zeta): H_{\Lambda}\to H_{\Lambda}$ that is analytic in $\zeta$ and
\begin{equation}
    R(0,\zeta)=A(\zeta)+\frac{\Pi_{\lambda}}{\lambda-\zeta}.
\end{equation}
Here $\Pi_{\lambda}: L^2\to L^2$ is the $L^2$ orthogonal projection onto the eigenspace of $\lambda$: $\Pi_{\lambda}u=\sum_{j=1}^m\langle u,v_{j} \rangle v_j$.

If we put
\begin{equation}
    \mathscr{E}(\zeta):=\begin{pmatrix} A(\zeta) & E_+ \\ E_- & E_{-+} \end{pmatrix}: H_{\Lambda}\times \CC^m \to H_{\Lambda}\times \CC^m,
\end{equation}
where $E_+: \CC^m\to H_{\Lambda}$, $E_-: H_{\Lambda}\times \CC^m$ are defined by 
\begin{equation}
\label{epm}
    E_+(v_+^1,\cdots, v_+^m)=\sum_{j=1}^m v_+^{j}u_j, \quad E_-v=(\langle v,u_1 \rangle_{L^2(\mathbb{T}^n)}, \cdots, \langle v,u_m \rangle_{L^2(\mathbb{T}^n)}),
\end{equation}
and $E_{-+}=(\zeta-\lambda)I_m$.
One can check that $\mathscr{E}(\zeta)$ is the inverse of $\mathscr{P}(0,\zeta)$. 
By \eqref{assumption2}, there exists $0<s_1<s_0$ such that for $s\in (-s_1,s_1)$, 
\begin{equation}
    \max\left(\|(P(s)-P)A(\zeta)\|_{H_{\Lambda}\to H_{\Lambda}}, \|A(\zeta)(P(s)-P)\|_{H_{\Lambda}\to H_{\Lambda}}\right)<1.
\end{equation}
Hence by \cite[Lemma C.3]{res}, for $(s,\zeta)\in (-s_1,s_1)\times U$, where $U$ is an open neighborhood of $0\in \CC$, $\mathscr{P}(s,\zeta)$ has inverse 
\begin{equation}
\label{ginv}
    \mathscr{E}(s,\zeta)=\begin{pmatrix} E(s,\zeta) & E_+(s,\zeta) \\ E_-(s,\zeta) & E_{-+}(s,\zeta) \end{pmatrix}
\end{equation}
such that $\mathscr{E}(s,\zeta)$ is analytic in $s,\zeta$ and $\mathscr{E}(0,\zeta)=\mathscr{E}(\zeta)$.
Since $P(s)-\zeta$ is invertible if and only if $E_{-+}(s,\zeta)$ is invertible, we know the resonances $\lambda_j(s)$, $1\leq j\leq m$, of $P(s)$ near $\lambda_0$ must satisfy 
\begin{equation}
    \textrm{det}(E_{-+}(s,\lambda_j(s)))=0, \quad \lambda_j(0)=\lambda.
\end{equation}
Let $L(s,\zeta):=\textrm{det}(E_{-+}(s,\zeta))$. Then $L(s,\zeta)$ is analytic in $(s,\zeta)\in (-s_1, s_1)\times U$, where $s_1>0$, $U$ is a neighborhood of $0\in \CC$, and
\begin{equation}
    L(0,\zeta)=\textrm{det}(E_{-+}(0,\zeta))=(\zeta-\lambda)^m.
\end{equation}
By Weierstrass preparation theorem (see for instance \cite[Theorem 8.2.15]{complex}), there exist analytic functions $g_{j}$, $g_{j}(\lambda)=0$, $0\leq j \leq m-1$, and analytic function $N(s,\zeta)$, $N(s,\zeta)\neq 0$ near $(0,\lambda)$, such that
\begin{equation}
    L(s,\zeta)=\left((\zeta-\lambda)^m+g_{m-1}(s)(\zeta-\lambda)^{m-1}+\cdots+g_0(s)\right)N(s,\zeta)
\end{equation}
Hence by \cite[Chapter 2, \S 1]{kato}, $\lambda_{j}$ has Puiseux series expansions. For a Puiseux cycle, there exists $p\in \NN$ such that
\begin{equation}
    \lambda_{\ell}(s)=\lambda+c_1\omega^{\ell}s^{1/p}+c_2\omega^{2\ell}s^{2/p}+\cdots, \quad 1\leq \ell \leq p
\end{equation}
where $\omega=e^{2\pi i/p}$. If $\lambda_{\ell}\in \RR$, then $p=1$ and $c_j\in \RR$, $j\geq 1$. If $\lambda_{\ell}\notin \RR$, then we have \eqref{mnr} using the fact that $\Im\lambda_{\ell}\leq 0$.

\noindent
\textbf{Part 2}.
Now we consider the case when $\lambda$ is a simple eigenvalue and prove the Fermi golden rule.

Differentiate $E_{-+}(s,\lambda(s))=0$ in $s$ and we find
\begin{equation*}
    \partial_{s}E_{-+}(0,\lambda)+\partial_{\zeta}E_{-+}(0,\lambda)\dot{\lambda}=0,
\end{equation*}
\begin{equation*}
\begin{split}
    & \partial_{s}^2E_{-+}(0,\lambda)+2\partial_{s}\partial_{\zeta}E_{-+}(0,\lambda)\dot{\lambda}
    + \partial_{\zeta}^2E_{-+}(0,\lambda)\dot{\lambda}^2 + \partial_{\zeta}E_{-+}(0,\lambda)\ddot{\lambda}=0.
\end{split}
\end{equation*}
Differentiate $\mathscr{P}(s,\zeta)\mathscr{E}(s,\zeta)=I$ in $s$ and we find
\begin{equation*}
    \partial_{s}E_{-+}(0,\lambda)=-E_-\dot{P}E_+, \quad \partial_s^2E_{-+}(0,\lambda)=2E_-\dot{P}E\dot{P}E_+-E_-\ddot{P}E_+.
\end{equation*}
Note now that $\partial_{\zeta}^2E_{-+}(0,\lambda)=0$, $\partial_{s}\partial_{\zeta}E_{-+}(0,\lambda)=0$, $\partial_{\zeta}E_{-+}(0,\lambda)=1$, $\dot{\lambda}(0)\in \RR$, hence we have
\begin{equation*}\begin{split}
    \Im\ddot{\lambda}
    = & -\Im\partial_{s}^2E_{-+}(0,\lambda) 
    =  -2\Im\langle A(\lambda)\dot{P}u, \dot{P}u \rangle+\Im\langle \ddot{P}u,u \rangle
    = -2\Im\langle A(\lambda)\dot{P}u,\dot{P}u \rangle.
\end{split}\end{equation*}
Here we used the self-adjointness of $\ddot{P}$.
Hence we find
\begin{equation}
    \Im\ddot{\lambda}=-2\Im\langle A(\lambda) \dot{P}u, \dot{P}u \rangle=-2\Im\langle \Pi_{\perp}R(\lambda)\Pi_{\perp}\dot Pu, \dot Pu \rangle.
\end{equation}
The latter equality follows from the fact that $A(\lambda)u=0$ and $A(\lambda)^*u=0$.
\end{proof}

\Remarks
1. We can derive an alternative formula for $\ddot{\lambda}(0)$ by using the operator $G_0^+$ constructed in \cite[Lemma 5.19]{scattering}:
\begin{equation}
\label{alt2nd}
    \ddot{\lambda}(0)=-4\pi^2\int_{\mathbb{S}^1}|G_0^+(R(\lambda)\Pi_{\perp}\dot{P}v_0)|^2dS.
\end{equation}
In fact, let $\mathcal{B}$ be defined by \cite[(6.2)]{scattering}. By the boundary pairing formula \cite[Proposition 6.5]{scattering}, we have
\begin{equation}\begin{split}
    & \langle (R(\lambda)-R(\lambda)^*)u,v \rangle 
    =  \langle R(\lambda)u,v \rangle-\langle u, R(\lambda)v \rangle\\
    = & \langle R(\lambda)u, (P-\lambda)R(\lambda)v \rangle-\langle R(\lambda)u, (P-\lambda)R(\lambda)v \rangle \\
    = & -\mathcal{B}(R(\lambda)u, R(\lambda)v)
    =  (2\pi)^2i\int_{\mathbb{S}^1}G_0^+(R(\lambda)u)\cdot G_0^+(R(\lambda)v)dS.
\end{split}\end{equation}
Here we used the fact $R(\lambda)u, R(\lambda)v\in I^0(\Lambda^+)$, see \cite[Lemma 4.1]{force} or \cite[Lemma 3.3]{scattering}. Hence by Theorem \ref{theorem1}, we have
\begin{equation}
    \begin{split}
        \ddot{\lambda}(0) & =-2\Im\langle \Pi_{\perp}R(\lambda)\Pi_{\perp}\dot{P}w,\dot{P}w \rangle = -\tfrac{1}{i}\langle (R(\lambda)-R(\lambda)^*)\Pi_{\perp}\dot{P}w, \Pi_{\perp}\dot{P}w. \rangle \\
       & = -4\pi^2\int_{\mathbb{S}^1}|G_0^+(R(\lambda)\Pi_{\perp}\dot{P}v_0)|^2dS.
    \end{split}
\end{equation}

\noindent
2. We can see from \eqref{alt2nd} that $\ddot{\lambda}(0)\leq 0$ and $\ddot{\lambda}(0)=0$ if and only if $\dot{P}v_0=cv_0$ for some $c\in \CC$. This implies the absence of eigenvalues for generic perturbations.

\section{Proof of Theorem 2}
\label{pf2}
In this section, we prove Theorem \ref{theorem3} by proposing a Grushin problem.
\begin{proof}[Proof of Theorem \ref{theorem3}]
As in \cite[(7.13)]{vis},  we put
\begin{equation}
    P_{q,t}:=P+it\Delta_{\mathbb{T}^n}-iQ, \quad Q:=S_{\Lambda}\Pi_{\Lambda}q\Pi_{\Lambda}T_{\Lambda}
\end{equation}
with $q\in C_c^{\infty}(\Lambda;[0,\infty))$ satisfies conditions in \cite[Lemma 7.6]{vis}. For the definition of $S_{\Lambda}$, $\Pi_{\Lambda}$, see \cite[\S 4, \S 5]{vis}. By \cite[Lemma 7.6]{vis}, for any $\epsilon>0$ sufficiently small, there exists $q=q(\epsilon)$ such that
\begin{equation}
    R_{q,t}:=(P_{q,t}-\zeta)^{-1}: H_{\Lambda}\to H_{\Lambda}
\end{equation}
exists for $t\in [0,t_0]$, $\zeta\in (-\zeta_0, \zeta_0)+i(-\theta\zeta_0,\infty)$. 
Note that
\begin{equation}
    R_{0,t}(\zeta)=(I+iR_{q,t}(\zeta)Q)^{-1}R_{q,t}(\zeta),
\end{equation}
hence the eigenvalues of $P(t)$ in $(-\zeta_0,\zeta_0)\times (-\theta\zeta_0,\infty)$ are values of $\zeta$ such that $I+iR_{q,t}(\zeta)Q:H_{\Lambda}\to H_{\Lambda}$ is not invertible. Since
\begin{equation}
    R_{0,0}(\zeta)=A(\zeta)+\frac{u\otimes u}{\lambda-\zeta},
\end{equation}
we have
\begin{equation}
    (I+iR_{q,0}(\zeta)Q)^{-1}=A(\zeta)(P_{q,0}-\zeta)+\frac{u\otimes u}{\lambda-\zeta}(P_{q,0}-\zeta)
\end{equation}

We now consider the Grushin problem
\begin{equation}
    \begin{pmatrix}I+iR_{q,t}(\zeta)Q & R_- \\ R_+ & 0\end{pmatrix}: H_{\Lambda}\times \CC \to H_{\Lambda}\times \CC,
\end{equation}
where $R_-: \CC \to H_{\Lambda}$, $R_+: H_{\Lambda}\to \CC$ are defined by
\begin{equation*}
    R_+w=\langle w, u \rangle_{L^2(\mathbb{T}^n)}, \quad R_-w_-= w_- R_{q,0}(\zeta)u.
\end{equation*}
As in the proof of Theorem \ref{theorem1}, this Grushin problem is wellposed and has an inverse
\begin{equation}
    \begin{pmatrix} E_{q,t}(\zeta) & E_{+,q,t}(\zeta) \\ E_{-,q,t}(\zeta) & E_{-+,q,t}(\zeta) \end{pmatrix}: H_{\Lambda}\times \CC \to H_{\Lambda}\times \CC
\end{equation}
with
$$E_{q,0}(\zeta)=A(\zeta)(P_{q,0}-\zeta), \ \ E_{-+,q,0}(\zeta)=\zeta-\lambda,$$ 
\begin{equation*}
    E_{-,q,0}(\zeta)v=\langle (P_{q,0}-\zeta)v,u \rangle_{L^2(\mathbb{T}^n)}, \quad E_{+,q,0}(\zeta)v_+=v_+u.
\end{equation*}
Now we have
\begin{equation}\begin{split}
    \dot{\lambda}
    = & -E_-\tfrac{d}{dt}(I+iR_{q,t}(\lambda)Q)E_+ 
    =  \langle (P_{q,0}-\lambda)R_{q,0}(\lambda)\Delta R_{q,0}(\lambda)Qu, u \rangle_{L^2(\mathbb{T}^n)} \\
    = & i\langle \Delta u, u \rangle_{L^2(\mathbb{T}^n)}=-i\|\nabla u\|_{L^2(\mathbb{T}^n)}^2.
\end{split}\end{equation}
This completes the proof.
\end{proof}

\end{document}